\documentclass[10pt]{article}

\usepackage[letterpaper]{geometry}
\usepackage{hicss51}
\usepackage{times}
\usepackage[none]{hyphenat}
\usepackage{url}
\usepackage{latexsym}
\usepackage{indentfirst}
\usepackage{graphicx}
\graphicspath{{images/}}

\usepackage{mathrsfs}
\usepackage{amssymb}
\usepackage{amsmath,bm}

\usepackage{mathrsfs}
\usepackage{graphicx}
\usepackage{eurosym}
\usepackage{amssymb}
\usepackage{amsmath}
\usepackage{amsfonts}
\usepackage{epstopdf}
\usepackage{epsf,subfigure}
\usepackage{psfrag}
\usepackage{graphics}
\usepackage{color} 
\usepackage{cite}
\usepackage{array,multirow,pbox}
\usepackage{enumitem}
\usepackage{caption}
\usepackage[USenglish]{babel}

\newcommand{\UL}[1]{\underline{#1}}
\newcommand{\OL}[1]{\overline{#1}}

\newcommand{\rem}[1]{}
\newtheorem{proposition}{Proposition}

\newtheorem{theorem}{Theorem}
\newtheorem{example}{Example}

\newtheorem{remark}{Remark}

\newenvironment{proof}[1][Proof]{\begin{trivlist}
\item[\hskip \labelsep {\bfseries #1}]}{\end{trivlist}}

\newcommand{\qed}{\nobreak \ifvmode \relax \else
      \ifdim\lastskip<1.5em \hskip-\lastskip
      \hskip1.5em plus0em minus0.5em \fi \nobreak
      \vrule height0.75em width0.5em depth0.25em\fi}

\graphicspath{{Figures/}}

\newcommand{\mysubeq}[2]{
\begin{subequations}\label{#1}
\begin{align}
#2
\end{align}\end{subequations}}

\title{Scalable Computation of 2D-Minkowski Sum of Arbitrary Non-Convex Domains: Modeling Flexibility in Energy Resources\footnote{}}

\author{Soumya Kundu, Vikas Chandan and Karan Kalsi\\
  Optimization and Control Group\\
  Pacific Northwest National Laboratory, Richland, WA 99352, USA\\
  {\underline{ \{soumya.kundu,\,vikas.chandan,\,karanjit.kalsi\}@pnnl.gov}} \\}

\date{}

\begin{document}
\maketitle
\begin{abstract}
The flexibility of active ($p$) and reactive power ($q$) consumption in distributed energy resources (DERs) can be represented as a (potentially non-convex) set of points in the $p$-$q$ plane. Modeling of the aggregated flexibility in a heterogeneous ensemble of DERs as a Minkowski sum (M-sum) is computationally intractable even for moderately sized populations. In this article, we propose a scalable method of computing the M-sum of the flexibility domains of a heterogeneous ensemble of DERs, which are allowed to be non-convex, non-compact. In particular, the proposed algorithm computes a guaranteed superset of the true M-sum, with desired accuracy. The worst-case complexity of the algorithm is computed. Special cases are considered, and it is shown that under certain scenarios, it is possible to achieve a complexity that is linear with the size of the ensemble. Numerical examples are provided by computing the aggregated flexibility of different mix of DERs under varying scenarios. 

\end{abstract}

\section{Introduction}

\vspace{-0.1in}

As the penetration of low-inertia renewable genera- tion increases, various forms of distributed energy resources (DERs), including flexible and responsive electrical loads, will be increasingly integrated into the grid operations to support ancillary services. In this context, DER refers to any load and distributed generation that can offer flexibility in net active and reactive power consumption (equivalently, generation). Real-time coordination and control of DERs requires an appropriate modeling and quantification of the loads behavior and their available flexibility. Modeling of aggregated flexibility in an ensemble of flexible loads for ancillary services (in particular, frequency regulation and ramping) have been explored in the literature in recent years \cite{Callaway:2009, Kundu:2012CDC, Perfumo:2012, Mathieu:2013, Zhang:2013, Hao:2015, Zhao:2016}. The proposed approaches are gen- erally applicable to ensembles of similar loads, such as a collection of residential air-conditioners, or a collection of plug-in electric vehicles. The aggregation of flexibility for a heterogeneous group of DERs, however, remains a challenging task. 

In order to efficiently coordinate tens of thousands of flexible loads in distribution systems, while also satisfying line-flow and node voltage constraints, hier- archical modeling and control frameworks have been proposed \cite{Callaway:2011,Bernstein:2015}. Hierarchical architectures alleviate the computational complexity of the optimal DER dispatch problem, by having resource aggregators to participate in the network optimization problem, instead of the individual DERs. Resource aggregators, referred to here as the \textit{aggregate device controller} (or ADC), are tasked with the responsibility of aggregating the flexibility, in active ($p$) and reactive power ($q$), of the neighboring DERs locally at the level of a couple of service transformers (tens of residential customers). The mix of such DERs is likely to be heterogeneous, e.g. a random collection of air-conditioners, electric water-heaters, batteries, solar photovoltaic inverers, wind inverters, etc. Flexibility of each DER can be represented as a union of points or domains in the $p$-$q$ plane. Aggregated $(p,q)$-flexibility of such an ensemble is a Minkowski sum (M-sum) of the flexibility sets of the individual DERs, and can be non-convex, non-compact. In recent work, \cite{Kundu:2018PSCC}, authors presented a geometric approach to approximate the aggregated flexibility domain of an ensemble of heterogeneous DERs using convex polygons. However, approximation of the true aggregated flexibility domain via computation of the exact M-sum of the individual flexibility domains still remains a computational challenge, with computational time increasing exponentially with the ensemble size.

Efficient computation of 2D M-sum of arbitrary polygons (convex/non-convex) have generated interest (primarily) in the field of computational geometry \cite{Ramkumar:1996,CGAL:2000,Agarwal:2002,Wein:2006,Behar:2011}, resulting in publicly available tools \cite{CGAL:2000,MPT:2004}. Convolution based methods \cite{Behar:2011,Wein:2006,CGAL:2000} have been shown to perform better than polynomial-decomposition based methods \cite{Agarwal:2002}. However, applicability and performance guarantees of such algorithms for M-sum of large number of arbitrary non-polygon 2D domains is not obvious. In this paper we propose a computationally tractable method of approximating the aggregated flexi- bility for an arbitrary collection of DERs. The rest of the paper is organized as follows. In Sec.\,\ref{S:problem} we describe the problem of aggregating flexibility of heterogeneous DERs. Sec.\,\ref{S:method} presents the key idea behind the proposed approach, while the detailed algorithmic steps along with an analysis of its computational performance is discussed in Sec.\,\ref{S:algo}. Numerical examples are presented in Sec.\,\ref{S:result}, before we conclude the article in Sec.\,\ref{S:concl}.

\section{Problem Description}\label{S:problem}

\vspace{-0.1in}

%
\begin{figure*}[h]
\centering
\captionsetup{justification=centering}
\subfigure[batteries]{
\includegraphics[scale=0.39]{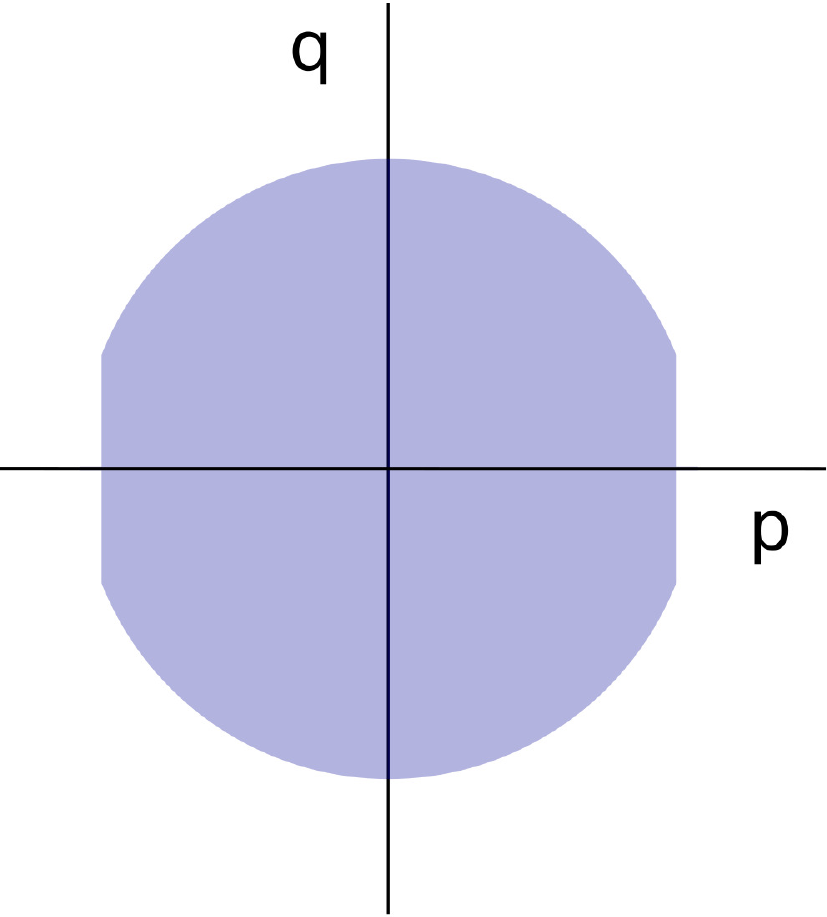}\label{F:battery}
}
\hspace{0.1in}
\subfigure[switching loads]{
\includegraphics[scale=0.39]{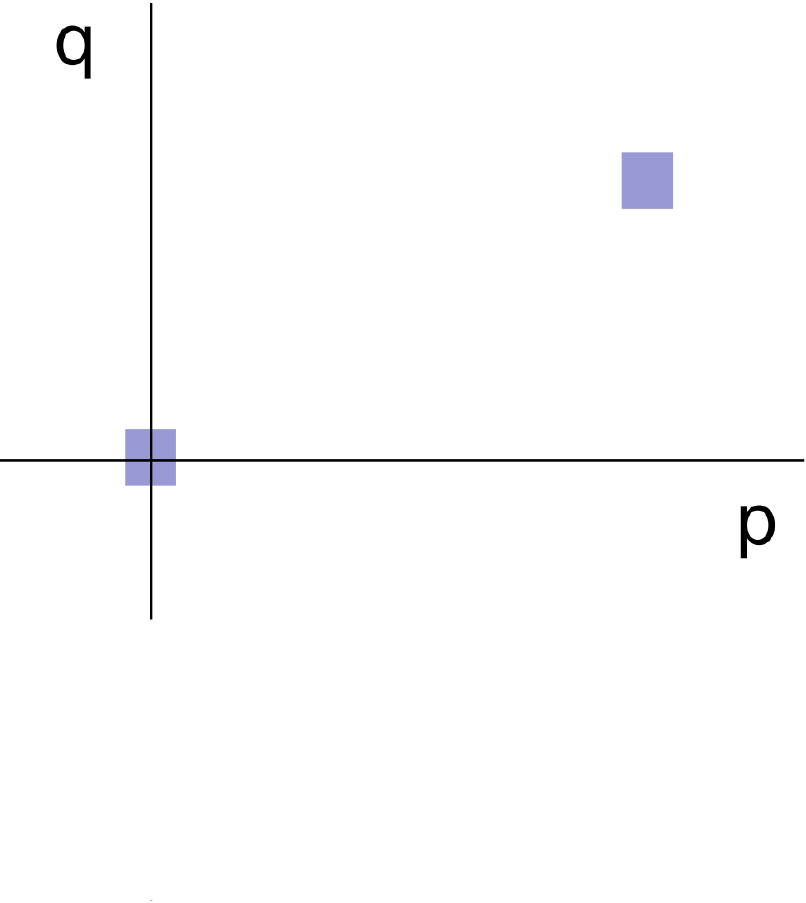}\label{F:hvac}
}
\hspace{0.1in}
\subfigure[PV inverters]{
\includegraphics[scale=0.39]{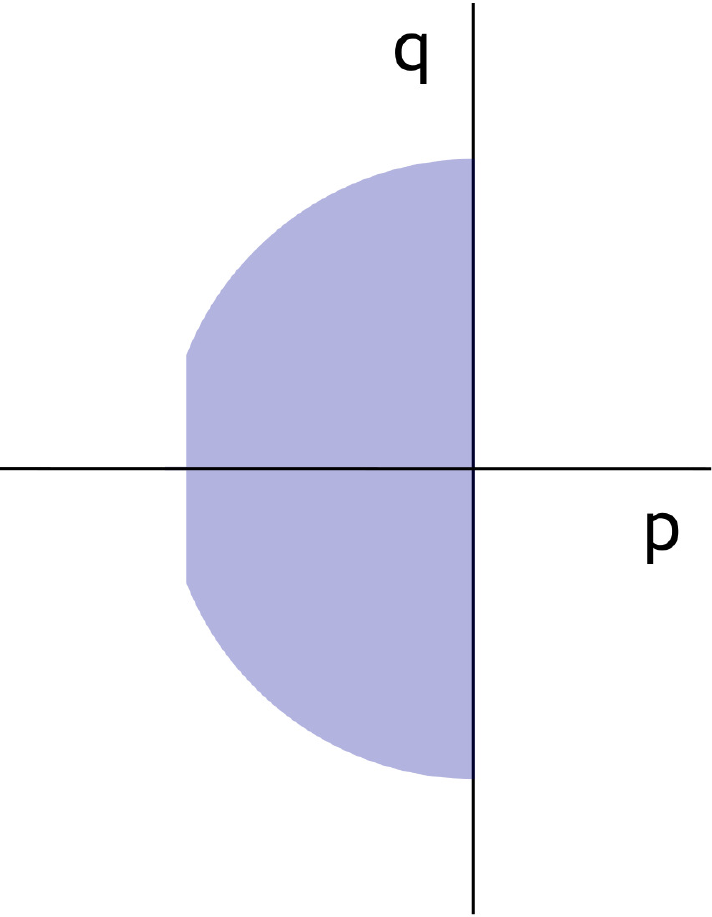}\label{F:pv}
}
\hspace{0.1in}
\subfigure[wind inverters]{
\includegraphics[scale=0.39]{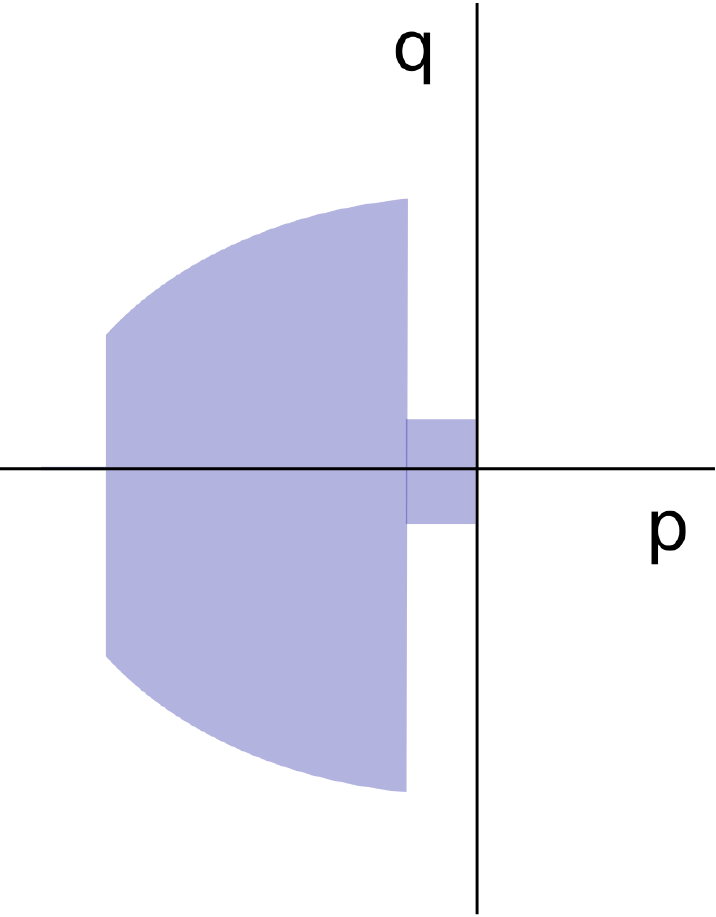}\label{F:wind}
}
\caption[Optional caption for list of figures]{Representation of continuous and discrete flexibility domains for an individual DER of certain types.}
\label{F:domains}
\end{figure*}
In the context of this article, any energy (consuming or generating) resource which offers certain flexibility in active ($p$) and/or reactive ($q$) power, possibly over a reasonably short time window (such as a 5-15\,min long control period), is considered as a DER. We use the notation $\mathcal{F}$ to represent the flexibility domain as a collection of $(p,q)$-points that are physically admissible by the DER (possibly via some local device-level control). Note that the flexibility domain ($\mathcal{F}$) could be a continuous or a discrete domain. For example, a solar photovoltaic (PV) inverter that can modulate its active and reactive power over a continuous range will have a continuous flexibility domain, while the flexibility domain for switching loads, such as an air-conditioner or electric water-heater, will be discrete. Moreover this flexibility is time-varying, and depends on exogenous parameters as well as end-user preferences. 

Figure\,\ref{F:domains} shows examples of flexibility domains for certain types of DERs, with positive (negative) values of $p$ and $q$ denoting consumption (generation). Discrete flexibility domain of a switching load (e.g. air-conditioner) that operates in two discrete operational states (`on' and `off') is shown in Figure\,\ref{F:hvac}, while the rest of the plots represent continuous flexibility domains. Batteries (Figure\,\ref{F:battery} offer full four-quadrant flexibility, while PV (Figure\,\ref{F:pv}) and wind (Figure\,\ref{F:wind}) inverters offer flexibility only on the left half-plane (active power generation). Note that the flexibility off- ered by the DERs is dynamic, and change based on end-usage and exogenous influence. For example, if there is a cloudy sky, the PV inverter output might only be restricted to a small fraction of its rated generation. Similarly, the air-conditioner may be operating mostly in `on'-state when the outside air-temperature is high. 

Let us assume that there are $N$ DERs within some heterogeneous ensemble. The flexibility domain of the $i$-th DER is denoted by $\mathcal{F}_i$\,, such that its active and reactive power consumption (with negative value signifying net generation)
\begin{align*}
(p_i,q_i)\in\mathcal{F}_i\,,\quad i\in\lbrace 1,2,\dots,N\rbrace\,.
\end{align*}
The goal of the flexibility aggregation task is to find the net flexibility domain $\mathcal{F}_{\Sigma}^{}$ in the form of a Minkowski  sum of the individual DER flexibility domains ($\mathcal{F}_i$), such that,
\begin{align*}
\mathcal{F}_\Sigma^{}&:=\biguplus_{i=1}^N\mathcal{F}_i=\left\lbrace (p,q)\,\left| \begin{array}{c} p=\sum_{i=1}^Np_i\\
q=\sum_{i=1}^Nq_i\\
\,(p_i,q_i)\in\mathcal{F}_i~\forall i\end{array}\right.\right\rbrace.
\end{align*}
Calculating the exact M-sum of the individual flexibility domains is computationally complex, especially as the number of DERs increases. Thus, from a practical point-of-view, a desirable approach is to construct approximations of the aggregate flexibility domain in a scalable way. In this paper, we propose a computationally tractable method to approximate, with arbitrary accuracy, the aggregate flexibility of the (heterogeneous) ensemble of DERs.

\section{Flexibility Aggregation: Key Idea}\label{S:method}

\vspace{-0.1in}

In general, the flexibility domains of DERs can be represented in the form of:
\mysubeq{E:flex_compact}{
\mathcal{F}&=\bigcup_{k=1}^K\mathcal{F}^k\\
\forall k:~\mathcal{F}^k&=\left\lbrace (p,q)\left|\begin{array}{c}
p\in\left[\UL{p}^k\!,\,\OL{p}^k\right]\\
q\in\left[\UL{g}^k(p),\,\OL{g}^k(p)\right]
\end{array}\right.\right\rbrace\!,\!
} 
where $K$ is some positive integer; and $\UL{g}^k(p)$ and $\OL{g}^k(p)$ represent the flexibility in reactive power at any given active power value within the range between $\UL{p}^k$ and $\OL{p}^k$. This representation can be used to represent the continuous and discrete flexibility domains  of the types depicted in Figure\,\ref{F:domains}, as illustrated below. 

\begin{example}
(\textsc{Batteries}) The flexibility domain of a battery, with a maximum charge/discharge rate of $p^{\max}$ and the apparent power rating of $s\!>\!p^{\max}$\,, is given by 
\begin{align*}
\mathcal{F}\!=\!\left\lbrace (p,q)\left|p\in[-p^{\max},{p}^{\max}],\,|q|\leq\sqrt{s^2-p^2}\right.\right\rbrace.
\end{align*} \end{example} 


\begin{example}
(\textsc{Wind Inverters}) The flexibility domain of a wind inverter, with a maximum active power generation of $p^{\max}$ and the apparent power ratings $s_1\!>\!\sqrt{\alpha}p^{\max}$ (due to rotor current limits) and $s_2\!>\!\sqrt{\alpha}p^{\max}$ (due to stator current limits)\,, is given by \cite{Lund:2007,Tian:2013,Martin:2015} 
\begin{align*}
\mathcal{F}&=\mathcal{F}^1\bigcup\mathcal{F}^2\bigcup\mathcal{F}^3\\
\mathcal{F}^1&=\left\lbrace (p,q)\left| \,p\in[-p^0,0]\,,\,q\in[-q^0,\,q^0]\right.\right\rbrace\\
\mathcal{F}^2&=\left\lbrace (p,q)\left|
\begin{array}{c}
p\in[-{p}^{\max},-p^0)\\
0\leq q\leq\sqrt{s_2^2-\alpha p^2}
\end{array}\right.\right\rbrace\\
\mathcal{F}^2&=\left\lbrace (p,q)\left|
\begin{array}{c}
p\in[-{p}^{\max},-p^0)\\
-\sqrt{s_1^2-\alpha p^2}\leq q\leq 0
\end{array}\right.\right\rbrace
\end{align*} 
where $p^0$ and $q^0$ are much smaller than the rated capacities, and $\alpha>0$\,. 
\end{example} 

\begin{example}\label{Ex:hvac}
(\textsc{Air-Conditioners}) The flexibility domain of a residential air-conditioner with an active power consumption rating of $p^{\max}$ (equal to the power consumed in `on' state) is represented by 
\begin{align*}
\mathcal{F}&=\mathcal{F}^1\cup\mathcal{F}^2\,,\,\mathcal{F}^1=\left\lbrace (0,0)\right\rbrace,\,\mathcal{F}^2=\left\lbrace ({p}^{\max},\gamma {p}^{\max})\right\rbrace
\end{align*} 
where $\gamma>0$ is related to the power factor. 
\end{example} 


Note that the individual flexibility domains can be continuous and discrete, as well as of arbitrary shape and size. Unfortunately, computation of the exact M-sum is possible only in some specific cases. For example, the M-sum of two (semi-)circles is another (semi-)circle with a radius equal to the sum of the radii of the individual (semi-)circles and its center at the vector sum of the centers of the individual (semi-)circles. As another example, the M-sum of two discrete (on/off) loads with individual flexibility domains $\lbrace (0,0),\,(p_1,q_1)\rbrace$ and $\lbrace (0,0),\,(p_2,q_2)\rbrace$\,, is given by $\lbrace (0,0),\,(p_1,q_1),\,(p_2,q_2),\,(p_1\!+\!p_2,q_1\!+\!q_2)\rbrace$\,. However, there does not exist a generic algo- rithm with guaranteed performance for computation of M-sum of continuous sets with arbitray shape (non-polygons). Moreover, the computational time re- quired to compute the M-sum of discrete on/off loads increases exponentially rendering the exact computation intractable even for  moderately sized populations. 

Therefore it is important to devise a scalable app- roach that computes a \textit{sufficiently close approximation} of the M-sum of diverse/heterogeneous set of individual domains, including both continuous and discrete do- mains. Let us first explain the key idea before going into the details of the algorithm in the next section. Note that it is easy to compute the exact M-sum of two domains each of which is modeled as a rectangle. Consider 
\begin{align*}
\mathcal{F}_1&=\left\lbrace (p,q)\left| \,p\in[\underline{p}_1,\overline{p}_1],\,q\in[\underline{q}_1,\overline{q}_1]\right.\right\rbrace\\
\mathcal{F}_2&=\left\lbrace (p,q)\left| \,p\in[\underline{p}_2,\overline{p}_2],\,q\in[\underline{q}_2,\overline{q}_2]\right.\right\rbrace.
\end{align*}
Note that we allow the boundaries of the rectangle to be non-unique, i.e. it is possible to have $\underline{p}_i=\overline{p}_i$ and/or $\underline{q}_i=\overline{q}_i$\,. Their M-sum is simply given by
\begin{align*}
\mathcal{F}_\Sigma^{}&=\mathcal{F}_1\cup\mathcal{F}_2\\
&=\left\lbrace (p,q)\left|\begin{array}{c}
p\in[\underline{p}_1+\underline{p}_2,\,\overline{p}_1+\overline{p}_2]\\
q\in[\underline{q}_1+\underline{q}_2,\,\overline{q}_1+\overline{q}_2]
\end{array} \right.\right\rbrace
\end{align*}
Extension of this to aggregation of domains which are unions of multiple ($\geq 1$) rectangles is possible as follows. Suppose
\begin{align*}
\mathcal{F}_1&=\bigcup_{i=1}^{m_1}\left\lbrace (p,q)\left| \,p\in[\underline{p}_1^i,\overline{p}_1^i],\,q\in[\underline{q}_1^i,\overline{q}_1^i]\right.\right\rbrace\\
\mathcal{F}_2&=\bigcup_{i=1}^{m_2}\left\lbrace (p,q)\left| \,p\in[\underline{p}_2^i,\overline{p}_2^i],\,q\in[\underline{q}_2^i,\overline{q}_2^i]\right.\right\rbrace,
\end{align*}
then
\begin{align*}
\mathcal{F}_\Sigma^{}&=\mathcal{F}_1\cup\mathcal{F}_2\\
&=\bigcup_{i=1}^{m_1}\bigcup_{j=1}^{m_2}\left\lbrace (p,q)\left|\begin{array}{c}
p\in[\underline{p}_1^i+\underline{p}_2^j,\,\overline{p}_1^i+\overline{p}_2^j]\\
q\in[\underline{q}_1^i+\underline{q}_2^j,\,\overline{q}_1^i+\overline{q}_2^j]
\end{array} \right.\right\rbrace
\end{align*}
The computational complexity of the above M-sum is $\mathcal{O}(m_1m_2)$\,. This forms the basis of our proposed algo- rithm for approximating M-sum of $N$ ($\geq 2$) flexibility domains of arbitrary size and shape. 

\begin{remark}
The above can be extended to shapes other than rectangles too, such as circles. However, in this paper, we will focus our attention to rectangles only.
\end{remark}

\section{Minkowski Sum: Scalable Algorithm}\label{S:algo}

\vspace{-0.1in}

Suppose that we have an ensemble of $N$ DERs with flexibility domains given by 
\mysubeq{E:flex_i}{
\mathcal{F}_i&=\bigcup_{k=1}^{K_i}\mathcal{F}_i^k\quad \forall i\in\lbrace 1,\dots,N\rbrace\\
\forall k:~\mathcal{F}_i^k&=\left\lbrace (p,q)\left|\begin{array}{c}
p\in\left[\UL{p}_i^k\!,\,\OL{p}_i^k\right]\\
q\in\left[\UL{g}_i^k(p),\,\OL{g}_i^k(p)\right]
\end{array}\right.\right\rbrace\!,\!
} 
Let us denote by $p_i^{\inf}$ and $p_i^{\sup}$, respectively, the \textit{infimum} and \textit{supremum} of the feasible $p$-points of the $i$-th DER, while $q_i^{\inf}$ and $q_i^{\sup}$ represent, respectively, the \textit{infimum} and \textit{supremum} of the feasible $q$-points, i.e.
\begin{subequations}\label{E:pqinfsup}
\begin{align}
p_i^{\inf}&:=\inf\lbrace p\left|\,\exists\, q\,\text{ so that }(p,q)\in\mathcal{F}_i\right.\rbrace\\
p_i^{\sup}&:=\sup\lbrace p\left|\,\exists\, q\,\text{ so that }(p,q)\in\mathcal{F}_i\right.\rbrace\\
q_i^{\inf}&:=\inf\lbrace q\left|\,\exists\, p\,\text{ so that }(p,q)\in\mathcal{F}_i\right.\rbrace\\
q_i^{\sup}&:=\sup\lbrace q\left|\,\exists\, p\,\text{ so that }(p,q)\in\mathcal{F}_i\right.\rbrace
\end{align}\end{subequations}

Note that for any given positive scalar $\varepsilon>0$\,, there exist positive integers $M_p(\varepsilon)$ and $M_q(\varepsilon)$ defined as:
\mysubeq{E:Mpq}{
M_p(\varepsilon):=\left\lceil \frac{1}{\varepsilon}\sum_{i=1}^N{\left(p_i^{\sup}\!-\!p_i^{\inf}\right)}\right\rceil\,.\\
M_q(\varepsilon):=\left\lceil\frac{1}{\varepsilon}\sum_{i=1}^N{\left(q_i^{\sup}\!-\!q_i^{\inf}\right)}\right\rceil\,.
}
If we were to split the whole $p$-range in the aggregate domain into equal-sized bins of length $\varepsilon$\,, we would need $M_p(\varepsilon)$ of such bins. Similarly to cover the $q$-range, we need $M_q(\varepsilon)$ bins. For reasons that will become clearer later, we will refer to $\varepsilon$ as the \textit{`tightness'} parameter - with smaller $\varepsilon$ implying \textit{tighter} results.

We are now in a position to describe the algorithm of approximate computation of M-sum. We will first describe two critical steps of the process, the \textit{initial discretization} step and the \textit{pixelization} step, and then move on the describe the complete algorithm.

\subsection{Initial Discretization Step}\label{S:discrete}

\vspace{-0.1in}

At the start of the process each of the individual flexibility domains $\mathcal{F}_i$ are discretized into sets of at most $M_p(\varepsilon)\times M_q(\varepsilon)$ rectangular blocks. There are several ways this discretization can be performed. To be specific, in this paper we consider  discretization of the following form:
\begin{align*}
\forall i\!:~\mathcal{F}_i&\mapsto\mathcal{F}_i^U(\varepsilon)=\bigcup_{j=1}^{m_i}\left\lbrace (p,q)\left| \begin{array}{c}
p\in[\underline{p}_i^j,\overline{p}_i^j]\\
q\in[\underline{q}_i^j,\overline{q}_i^j]
\end{array}\right.\right\rbrace \\
\text{where,}~&  \overline{p}_i^j-\underline{p}_i^j\leq \frac{{\varepsilon\left(p_i^{\sup}\!-\!p_i^{\inf}\right)}}{\sum_{i=1}^N{\left(p_i^{\sup}\!-\!p_i^{\inf}\right)}}\quad\forall j\!=\!1,\dots,m_i\\
& \overline{q}_i^j-\underline{q}_i^j\leq \frac{{\varepsilon\left(q_i^{\sup}\!-\!q_i^{\inf}\right)}}{\sum_{i=1}^N{\left(q_i^{\sup}\!-\!q_i^{\inf}\right)}}\quad\forall j\!=\!1,\dots,m_i\\
&m_i\leq M_p(\varepsilon)\cdot M_q(\varepsilon)\\
\text{and }~& \mathcal{F}_i\subseteq\mathcal{F}_i^U(\varepsilon)\,.
\end{align*}
Note that, for example, since the discrete switching (on/off) loads have only two discrete points in the feasibility space, for those loads $\mathcal{F}_i=\mathcal{F}_i^U(\varepsilon)$\,. 


\begin{remark}
The last condition ensures that the discrete approximation is a superset of the corresponding actual flexibility domain. This condition can be modified to reflect a subset or any other type of approximation.
\end{remark}

\subsection{Pixelization Step}\label{S:pixel}

\vspace{-0.1in}

Let us consider the M-sum of two domains each of which is represented as a union of at most $M_p(\varepsilon)\times M_q(\varepsilon)$ rectangular blocks, as follows:
\begin{align*}
\forall i\in\lbrace 1,2\rbrace:~\mathcal{F}_i^U(\varepsilon)=&\,\bigcup_{j=1}^{m_i}\left\lbrace (p,q)\left| \begin{array}{c}
p\in[\underline{p}_i^j,\overline{p}_i^j]\\
q\in[\underline{q}_i^j,\overline{q}_i^j]
\end{array}\!\right.\!\!\right\rbrace\!\\
\forall j\!=\!1,\dots,m_i\!:~\overline{p}_i^j-\underline{p}_i^j&\leq \frac{{\varepsilon\left(p_i^{\sup}\!-\!p_i^{\inf}\right)}}{\sum_{i=1}^N{\left(p_i^{\sup}\!-\!p_i^{\inf}\right)}}\\
\overline{q}_i^j-\underline{q}_i^j&\leq \frac{{\varepsilon\left(q_i^{\sup}\!-\!q_i^{\inf}\right)}}{\sum_{i=1}^N{\left(q_i^{\sup}\!-\!q_i^{\inf}\right)}}\\
m_i&\leq M_p(\varepsilon)\cdot M_q(\varepsilon)\,.
\end{align*}
We want to approximate their M-sum $\mathcal{F}_1^U(\varepsilon)\cup \mathcal{F}_2^U(\varepsilon)$ as another union of at most $M_p(\varepsilon)\times M_q(\varepsilon)$ rectangular blocks, denoted by $\mathcal{F}_{\Sigma}^U$\, such that
\begin{align*}
\mathcal{F}_{\Sigma}^U&\supseteq \mathcal{F}_1^U(\varepsilon)\cup \mathcal{F}_2^U(\varepsilon)\\
&=\bigcup_{i=1}^{m_1}\bigcup_{j=1}^{m_2}\left\lbrace (p,q)\left|\begin{array}{c}
p\in[\underline{p}_1^i+\underline{p}_2^j,\,\overline{p}_1^j+\overline{p}_2^j]\\
q\in[\underline{q}_1^i+\underline{q}_2^j,\,\overline{q}_1^i+\overline{q}_2^j]
\end{array} \right.\right\rbrace.
\end{align*}
Clearly, 
\begin{align*}
m_1\,m_2\leq M_p(\varepsilon)\!\cdot\!M_q(\varepsilon)\implies \mathcal{F}_{\Sigma}^U= \mathcal{F}_1^U(\varepsilon)\cup \mathcal{F}_2^U(\varepsilon)\,.
\end{align*} 
When, however, $m_1\,m_2> M_p(\varepsilon)\cdot M_q(\varepsilon)$\,, we propose the following two steps for appropriate pixelization of the M-sum $\mathcal{F}_1^U(\varepsilon)\cup \mathcal{F}_2^U(\varepsilon)$\,.

\subsubsection{\textsc{(Pixelization) Step 1:} }

We first compute the estimated number of \textit{pixels} on the $p$- and $q$-axes. In order to maintain at most $M_p(\varepsilon)\times M_q(\varepsilon)$ pixels, we use the \textit{pixel sizes} on the $p$-axis and $q$-axis as:
\mysubeq{}{
\widehat{\varepsilon}_p &= \varepsilon\,\frac{\sum_{i=1}^2 \left(p_i^{\sup}-p_i^{\inf}\right)}{\sum_{i=1}^N{\left(p_i^{\sup}\!-\!p_i^{\inf}\right)}}\leq\varepsilon\\
\widehat{\varepsilon}_q &= \varepsilon\,\frac{\sum_{i=1}^2 \left(q_i^{\sup}-q_i^{\inf}\right)}{\sum_{i=1}^N{\left(q_i^{\sup}\!-\!q_i^{\inf}\right)}}\leq \varepsilon
}
With the appropriate pixel-sizes $\widehat{\varepsilon}_p$ and $\widehat{\varepsilon}_q$ determined, the rectangular space defined by 
\begin{align*}
\left\lbrace (p,q)\left|\begin{array}{c}
p\in[\underline{p}_1^{\inf}+\underline{p}_2^{\inf},\,\overline{p}_1^{\sup}+\overline{p}_2^{\sup}]\\
q\in[\underline{q}_1^{\inf}+\underline{q}_2^{\inf},\,\overline{q}_1^{\sup}+\overline{q}_2^{\sup}]
\end{array} \right.\right\rbrace,
\end{align*}
which is a superset of the aggregated flexibility space $\mathcal{F}_1^U(\varepsilon)\cup \mathcal{F}_2^U(\varepsilon)$\,, can be represented with the help of a total of $M_p(\varepsilon)\times M_q(\varepsilon)$ equal-sized pixels, $M_p(\varepsilon)$ pixels on the $p$-axis and $M_q(\varepsilon)$ pixels on the $q$-axis. Note that
\begin{align*}
\left\lceil{\frac{1}{\widehat{\varepsilon}}\sum_{i=1}^2 \left(p_i^{\sup}-p_i^{\inf}\right)}\right\rceil&=M_p(\varepsilon)\notag\\
\left\lceil{\frac{1}{\widehat{\varepsilon}}\sum_{i=1}^2 \left(q_i^{\sup}-q_i^{\inf}\right)}\right\rceil&=M_q(\varepsilon)\,.\notag
\end{align*}
The domain defined by each $(k,l)$-th pixel, where $k\in\lbrace 1,\dots,M_p(\varepsilon)\rbrace$ and $l\in\lbrace 1,\dots,M_q(\varepsilon)\rbrace$\,, is defined by:
\begin{align}\label{E:pixel}
&\mathbf{pixel}(k,l)\\
&:=\left\lbrace (p,q)\left|\begin{array}{c}
(k-1)\,\widehat{\varepsilon}_p\leq p-\left(\underline{p}_1^{\inf}+\underline{p}_2^{\inf}\right)\leq k\,\widehat{\varepsilon}_p\\
(l-1)\,\widehat{\varepsilon}_q\leq q-\left(\underline{q}_1^{\inf}+\underline{q}_2^{\inf}\right)\leq l\,\widehat{\varepsilon}_q\end{array}\!\!\right.\!\!\right\rbrace\notag
\end{align}

\subsubsection{\textsc{(Pixelization) Step 2:}}

As the final step in the \textit{pixelization} step, we identify the \textit{pixels} that have some overlap with the sum $\mathcal{F}_1^U(\varepsilon)\cup \mathcal{F}_2^U(\varepsilon)$\,. Recall that the sum is a union of $m_1\,m_2$ rectangular blocks. Thus, for every pair of $(i,j)\,,\,i\in\lbrace 1,\dots,m_1\rbrace\,,\,j\in\lbrace 1,\dots,m_2\rbrace$\,, there is a rectangular block in the sum $\mathcal{F}_1^U(\varepsilon)\cup \mathcal{F}_2^U(\varepsilon)$ defined by 
\begin{align*}
\left\lbrace (p,q)\left|\begin{array}{c}
p\in[\underline{p}_1^i+\underline{p}_2^j,\,\overline{p}_1^j+\overline{p}_2^j]\\
q\in[\underline{q}_1^i+\underline{q}_2^j,\,\overline{q}_1^i+\overline{q}_2^j]
\end{array} \right.\right\rbrace.
\end{align*}
Associated with it are the following indexes:
\mysubeq{E:index}{
k_0&=\max\left\lbrace 1,\left\lceil{\frac{1}{\widehat{\varepsilon}_p} \left(\underline{p}_1^i+\underline{p}_2^j-p_1^{\inf}-p_2^{\inf}\right)}\right\rceil\right\rbrace\\
k_f&=\max\left\lbrace 1,\left\lceil{\frac{1}{\widehat{\varepsilon}_p} \left(\overline{p}_1^i+\overline{p}_2^j-p_1^{\inf}-p_2^{\inf}\right)}\right\rceil\right\rbrace\\
l_0&=\max\left\lbrace 1,\left\lceil{\frac{1}{\widehat{\varepsilon}_q} \left(\underline{q}_1^i+\underline{q}_2^j-q_1^{\inf}-q_2^{\inf}\right)}\right\rceil\right\rbrace\\
l_f&=\max\left\lbrace 1,\left\lceil{\frac{1}{\widehat{\varepsilon}_q} \left(\overline{q}_1^i+\overline{q}_2^j-q_1^{\inf}-q_2^{\inf}\right)}\right\rceil\right\rbrace
}
such that the union of all the $\mathbf{pixel}(k,l)$ with $k\in[k_0,k_f]$ and $l\in[l_0,l_f]$ is a superset and an appropriate \textit{pixelization} of the corresponding rectangular block in the sum $\mathcal{F}_1^U(\varepsilon)\cup \mathcal{F}_2^U(\varepsilon)$\,. Repeating step 2 for every pair of $(i,j)$ in the M-sum, we readily obtain the appropriately pixelated superset of $\mathcal{F}_1^U(\varepsilon)\cup \mathcal{F}_2^U(\varepsilon)$\,.

\begin{remark}
The computational efforts in step 2 for the pixelization are minimal, since it requires only algebraic computation of the indexes \eqref{E:index} to obtain the overlap. Moreover, the step 1 is to be done only once for every sum $\mathcal{F}_1^U(\varepsilon)\cup \mathcal{F}_2^U(\varepsilon)$\,, and involve only minimal algebraic calculations. Therefore the computational complexity of the steps 1-2 combined is $\mathcal{O}(m_1m_2)$\,.
\end{remark}

\subsection{Complete Algorithm}\label{S:summary}

\vspace{-0.1in}

We are finally in a position to state the complete algorithm for scalable computation of the close superset of the M-sum of a heterogeneous collection of $N$ DERs,
\begin{align}
\mathcal{F}_\Sigma^U\supseteq\mathcal{F}_\Sigma:=\bigcup_{i=1}^N\mathcal{F}_i\,.
\end{align} 
The idea is to perform an $N$-step iterative process to compute the M-sum, where each step $t\in\lbrace 0,\dots,N-1\rbrace$ (we denote $t=0$ as the initial step) involves the computation of M-sum of a pair of domains. Let us de- note by $\mathcal{F}_\Sigma^{U(t+1)}$ the approximation of the M-sum $\mathcal{F}_\Sigma^{}$\,, at the completion of step-$t$\,, such that
\begin{align*}
\mathcal{F}_\Sigma^{U(1)}\subseteq\mathcal{F}_\Sigma^{U(2)}\subseteq\dots\subseteq\mathcal{F}_\Sigma^{U(N)}:=\mathcal{F}_\Sigma^U\,.
\end{align*} 
The algorithmic steps are outlined as below:

\subsubsection{\textsc{(Minkowski) Initialization Step-0:}}

At this step, we perform two tasks. The first task is to discretize each flexibility domain $\mathcal{F}_i$ into $\mathcal{F}_i^U(\varepsilon)$\,, a set of at most $M_p(\varepsilon)\times M_q(\varepsilon)$ rectangular blocks, as described in Sec.\,\ref{S:discrete}. Moreover we initialize $\mathcal{F}_\Sigma^{U(1)}$ as $\mathcal{F}_1^U(\varepsilon)$ (where the DERs are sorted in no particular order), i.e. we perform  the following:
\mysubeq{}{
\textbf{compute}\quad &\mathcal{F}_i^U(\varepsilon)\supseteq \mathcal{F}_i ~\forall i\,.\\
\textbf{initialize}\quad &\mathcal{F}_\Sigma^{U(1)}\gets \mathcal{F}_1^U(\varepsilon)\,.
}
Note that, by construction, each of $\mathcal{F}_i^U(\varepsilon)\,\forall i$ and $\mathcal{F}_\Sigma^{U(1)}$ is represented by the union of at most $M_p(\varepsilon)\times M_q(\varepsilon)$ rectangular blocks.

\subsubsection{\textsc{(Minkowski) Iterative Step}-$t\, (\!\geq\!1)$:}

In the following iterative steps-$t\,,\,t\!=\!1,2,\dots,N\!-\!1$\,, we perform the \textit{pixelization} on the M-sum of $\mathcal{F}_{t+1}^U(\varepsilon)$ and $\mathcal{F}_\Sigma^{U(t)}$\,. Applying the two-steps \textit{pixelization} procedure, outlined in Sec.\,\ref{S:pixel}, on the M-sum of $\mathcal{F}_{t+1}^U(\varepsilon)$ and $\mathcal{F}_\Sigma^{U(t)}$\,, we obtain 
\begin{align*}
\textbf{pixelize}\quad & \mathcal{F}_\Sigma^{U(t+1)}\supseteq \mathcal{F}_{t+1}^U(\varepsilon)\cup \mathcal{F}_\Sigma^{U(t)}
\end{align*}
By construction, $\mathcal{F}_\Sigma^{U(t+1)}$ is also a union of at most $M_p(\varepsilon)\!\times\! M_q(\varepsilon)$ rectangular blocks. Repeat until $t\!=\!N\!-\!1$. 
\begin{remark}
It is unnecessary to perform the pixelization at the final step ($t=N-1$) at which a simple M-sum of $\mathcal{F}_N^U(\varepsilon)$ and $\mathcal{F}_\Sigma^{U(N-1)}$ suffices.
\end{remark}

At the conclusion of $(N\!-\!1)$-th step, we set 
\begin{align}
\mathcal{F}_\Sigma^{U}\gets \mathcal{F}_\Sigma^{U(N)}\,.
\end{align}

\subsection{Analysis}

\vspace{-0.1in}

Recall that, by construction (enforced by both the \textit{initial discretization} as well as the \textit{pixelization} steps), the estimated flexibility domain ($\mathcal{F}_\Sigma^{U}$) is a guaranteed superset of the true M-sum ($\mathcal{F}_\Sigma^{}$). Moreover, by def- inition of the \textit{tightness} parameter $\varepsilon$\,, for every point $(p^*,q^*)$ in the calculated M-sum, there always exists a feasible point $(p,q)$ with $p\in[p^*\!-\!\varepsilon,p^*\!+\!\varepsilon]$  and $q\in[q^*\!-\!\varepsilon,q^*\!+\!\varepsilon]$\,, i.e. within the Euclidean distance of $\sqrt{2}\,\varepsilon$ from $(p^*,q^*)$. Furthermore, because of the \textit{pixelization} step, the \textit{worst-case} computational complexity of each iteration step is $\mathcal{O}\left(M_p(\varepsilon)^2 M_q(\varepsilon)^2\right)$\,, which brings the \textit{worst-case} complexity of the complete algorithm down to $\mathcal{O}\left(NM_p(\varepsilon)^2 M_q(\varepsilon)^2\right)$\,. The performance of the algo- rithm is summarized in the following result:

\begin{theorem}\label{T:worst_case}
(\textsc{Worst-case Complexity}) For any tightness parameter $\varepsilon$\,, the algorithm in Sec.\,\ref{S:summary} gen- erates a guaranteed superset of the true M-sum of $N$ arbitrary shaped 2D-domains, such that for every point $(p^*\!,q^*)$ in the calculated M-sum, there always exists a feasible point $(p,q)$ with 
\begin{align*}
\max\left(\left|p\!-\!p^*\right|,\left|q\!-\!q^*\right|\right)\!\leq\!\varepsilon\,.
\end{align*} 
Moreover, the worst-case complexity is given by $\mathcal{O}\left(NM_p(\varepsilon)^2 M_q(\varepsilon)^2\right)$\,.
\end{theorem}

While Theorem\,\ref{T:worst_case} describes the worst-case complex- ity of the proposed algorithm in the generic case, we discuss some specific scenarios in the rest of this section.

\subsubsection{\textsc{Special Case I}:}

Consider the case when the DERs in the ensemble are similarly sized, i.e. 
\mysubeq{E:homo}{
\forall i\,:\quad p_i^{\sup}-p_i^{\inf} &= \Delta p & \quad\text{(i.e. uniform)}\\
q_i^{\sup}-q_i^{\inf} &= \Delta q & \quad\text{(i.e. uniform)}
}
\begin{proposition}\label{P:uniform}
In the specific case outlined in \eqref{E:homo}, the computational complexity is $\mathcal{O}\left(N^3/\varepsilon^4\right)$\,.
\end{proposition}

\begin{proof}
According  to the assumption of \textit{uniformity}, the flexibility domain of each DER in the ensemble is discretized into $\Delta p/\varepsilon$-bins on $p$-axis and $\Delta q/\varepsilon$-bins on $q$-axis, i.e. a total of $\left(\Delta p\,\Delta q\right)/\varepsilon^2$ rectangular blocks (or, pixels). The M-sum of first two DERs (in no particular order) would result in (at most) $\left(2\Delta p/\varepsilon\right)\times \left(2\Delta q/\varepsilon\right)$-pixels. It is not difficult to see that the M-sum of first $n< N$ DERs can be similarly represented by $n^2\left(\Delta p\,\Delta q\right)/\varepsilon^2$ pixels. Note that, as per \eqref{E:Mpq} and \eqref{E:homo}
\begin{align*}
M_p(\varepsilon)\,M_q(\varepsilon)\geq n^2\left(\Delta p\,\Delta q\right)/\varepsilon^2\quad\forall n\leq N\,.
\end{align*}
Complexity of the computation of the incremental M- sum of the $(n+1)$-th DER and the M-sum of the first $n$ DERs is proportional to $n^2/\varepsilon^4$\,. Complexity of computing the M-sum of $N$ DERs is then proportional to $\sum_{i=1}^Nn^2/\varepsilon^4$, resulting in a complexity of $\mathcal{O}\left(N^3/\varepsilon^4\right)$\,.\hfil\hfill\qed
\end{proof}

It can be argued that the above result holds for many practical scenarios when the ensemble has similarly sized DERs, while guaranteeing a desired accuracy.

\subsubsection{\textsc{Special Case II}:}

In certain scenarios, due to limitations on the computing resources, it might be necessary to put a bound on the number of bins used to pixelize the $p$- and $q$-axes. For example, let us consider the scenario when the number of bins on the $q$-axis is fixed to $\overline{M}_q$ while the bins on the $q$-axis are upper bounded by $\overline{M}_p$\,, i.e.
\begin{align}\label{E:halfbound}
M_q = \overline{M}_q\,\text{ but }M_p\leq\overline{M}_p\,,
\end{align}

\begin{proposition}\label{P:bounded}
In the specific case when \eqref{E:halfbound} and \eqref{E:homo} hold, the complexity is: 1) $\mathcal{O}(N^2\!/\varepsilon^2)$ for sufficient- ly large $\overline{M}_p$\,, and
2) $\mathcal{O}(N/\varepsilon)$ when $\overline{M}_p$ is small\,. Moreover, the tightness of the calculated M-sum deteriorates linearly with the size of ensemble.
\end{proposition}

\begin{proof}
Follows similarly as the one for Proposition\,\ref{P:uniform}. Note that the flexibility domain of each DER can be discretized in a total of $\left(\Delta p\,\overline{M}_q\right)\!/\varepsilon$ pixels. The M- sum of first $n$ DERs can be represented by 
\begin{align*}
\left\lbrace\begin{array}{c}{n\left(\Delta p\,\overline{M}_q\right)} /{\varepsilon} \,\text{ pixels if }n<\overline{M}_p\,\varepsilon/\Delta p\\
\\
\text{and }\overline{M}_p \,\text{ pixels if otherwise.}\end{array}\right.
\end{align*} 
Therefore, for sufficiently large $\overline{M}_p$\,, the complexity of computing the M-sum of N DERs can be shown to be proportional to $\sum_{i=1}^Nn/\varepsilon^2$\,, resulting in a complexity of $\mathcal{O}\!\left(N^2\!/\varepsilon^2\right)$\,. On the other hand, for small $\overline{M}_p$\,, the complexity of computing the M-sum of N DERs can be shown to be proportional to $\sum_{i=1}^N1/\varepsilon$\,, resulting in a complexity of $\mathcal{O}\!\left(N\!/\varepsilon\right)$\,. For the tightness argument, note from \eqref{E:Mpq} and \eqref{E:homo},
\begin{align*}
\varepsilon\geq \max\left(N\Delta p/\overline{M}_p,\,N\Delta q/\overline{M}_q\right)\,,
\end{align*}
i.e. for every point in the calculated M-sum, there exists a feasible point within some Euclidean distance that scales linearly with $N$\,.\hfill\hfill\qed
\end{proof}

%
%
%

\section{Numerical Results}\label{S:result}

\vspace{-0.1in}

\begin{figure}[h]
\centering
\captionsetup{justification=centering}
\subfigure[a population of 5 air-conditioners]{
\includegraphics[scale=0.42]{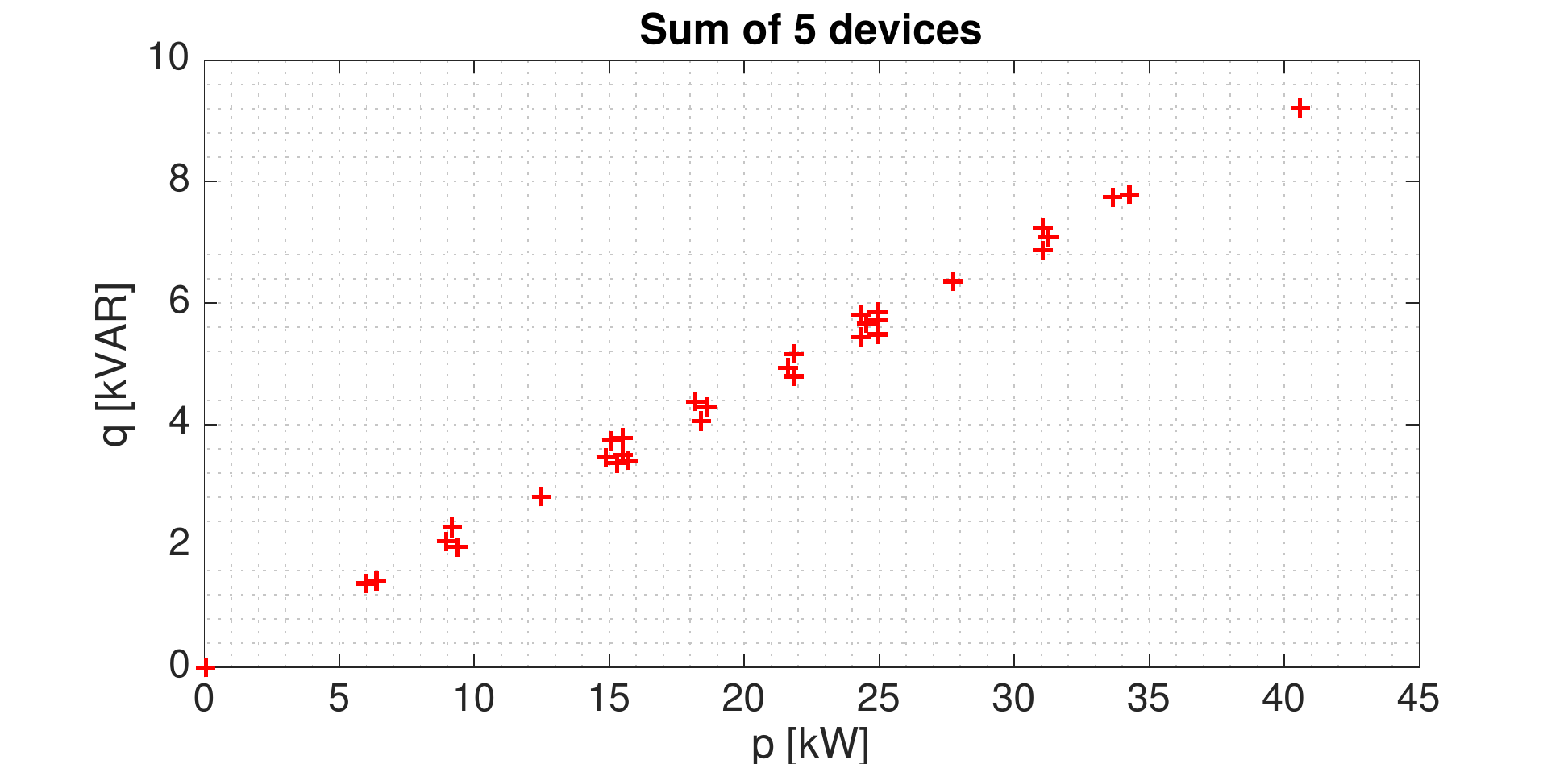}\label{F:hvac5}
}
\hspace{0.1in}
\subfigure[a population of 3 wind inverters]{
\includegraphics[scale=0.42]{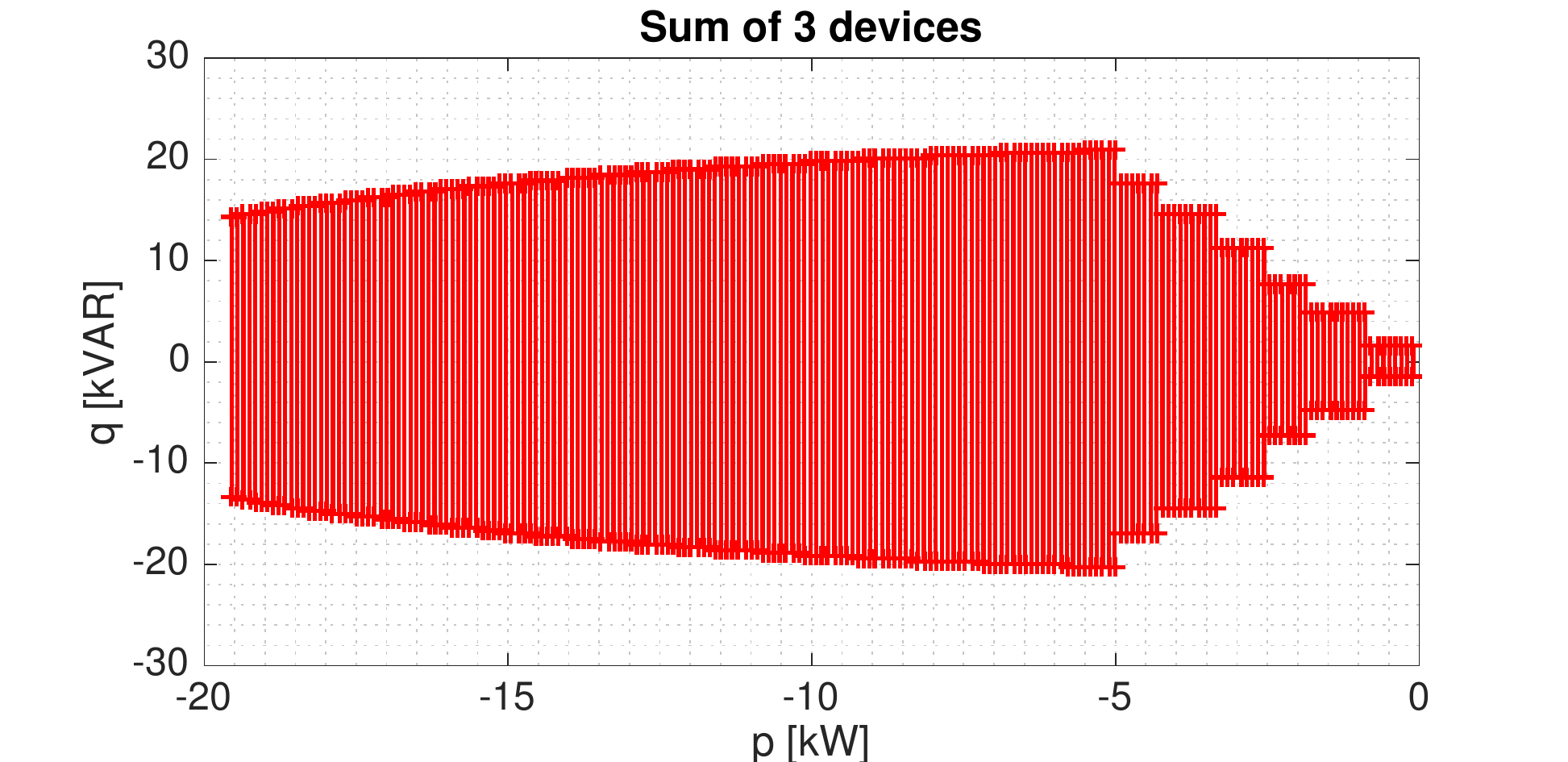}\label{F:wind3}
}
\hspace{0.1in}
\subfigure[a total of 5 PVs and water heaters]{
\includegraphics[scale=0.42]{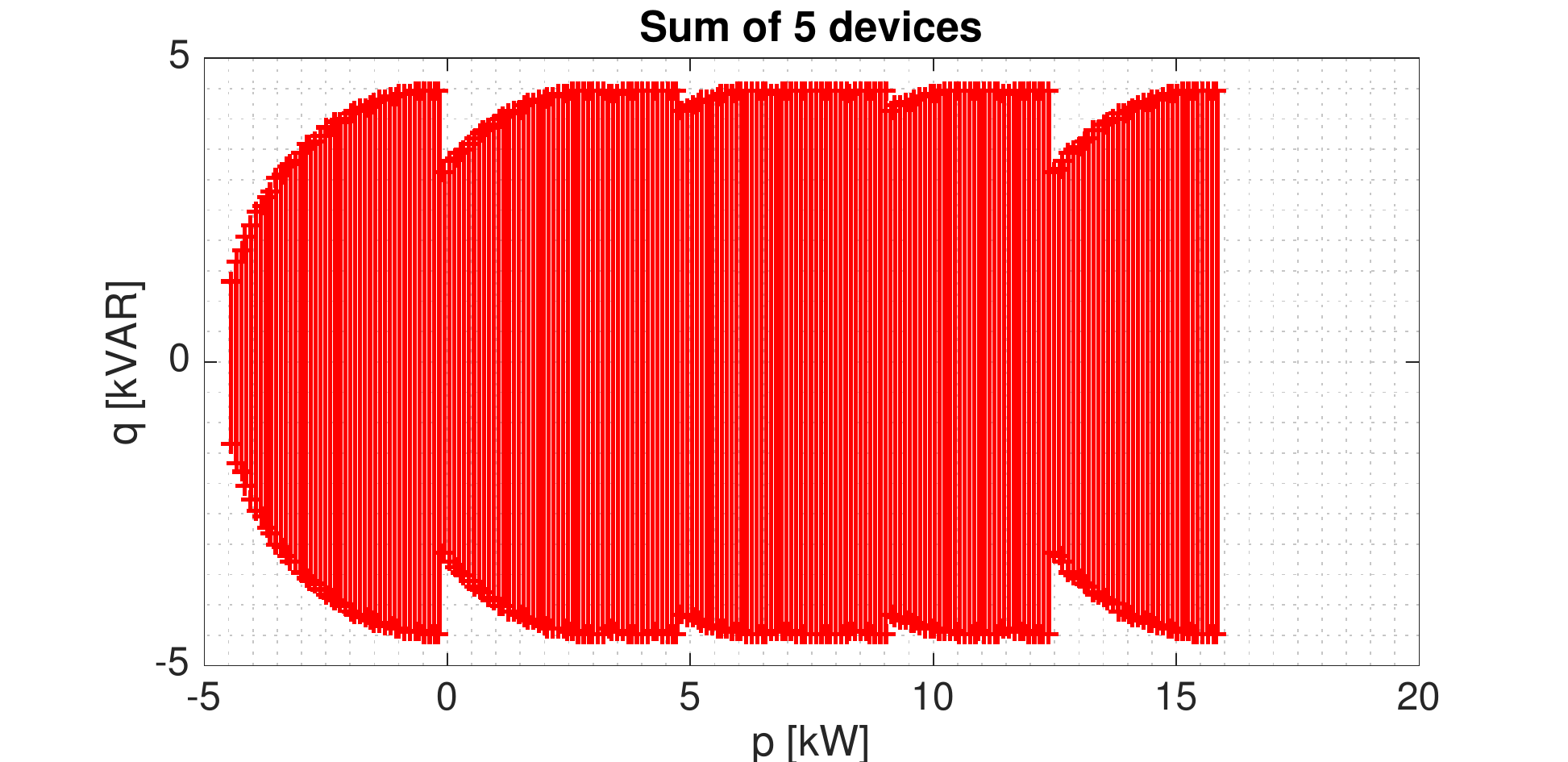}\label{F:ewh_PV_5}
}
\caption[Optional caption for list of figures]{Aggregated flexibility domains for various arbitrary sample cases.}
\label{F:sample}
\end{figure}

We start by illustrating in Figure\,\ref{F:sample} how M-sum looks like in some arbitrarily generated ensemble of DERs (including air-conditioners, wind inverters, electric water- heaters and solar photovoltaic inverters). In order to better demonstrate the dependence of the computational complexity of the proposed algorithm on the ensemble size and the tightness parameter, we run a set of tests. First we randomly choose a 10 DER ensemble with of one air-conditioners, two water-heaters, two batteries, two wind inverters and three photovoltaic inverters. This population is then replicated multiple times to create ensembles of larger size (in the multiples of 10). Moreover, the tightness parameter $\varepsilon$ is varied from $0.01$ kW (or, kVAR) to $0.64$ kW (or, kVAR). The computation times are calculated by running the algorithm for various combinations of ensemble size and tightness parameters. Moreover, we also run the tests for two different cases: one in which $\overline{M}_p$ (in Proposition\,\ref{P:bounded}) is chosen to be 600, and another in which it is chosen to be 4000 (very high). Figures\,\ref{F:600} and \ref{F:Inf} present the test results. It can be seen that the results align with the analysis summarized in Proposition\,\ref{P:bounded}. Specifically, the computational complexity is seen to be roughly linear with the ensemble size when $\overline{M}_p$ is set to 600 (small), while it is $\mathcal{O}(N^2)$ when $\overline{M}_p$ is set to 4000 (large).

\begin{figure*}[h]
\centering
\captionsetup{justification=centering}
\subfigure[computation time vs. tightness parameter]{
\includegraphics[scale=0.42]{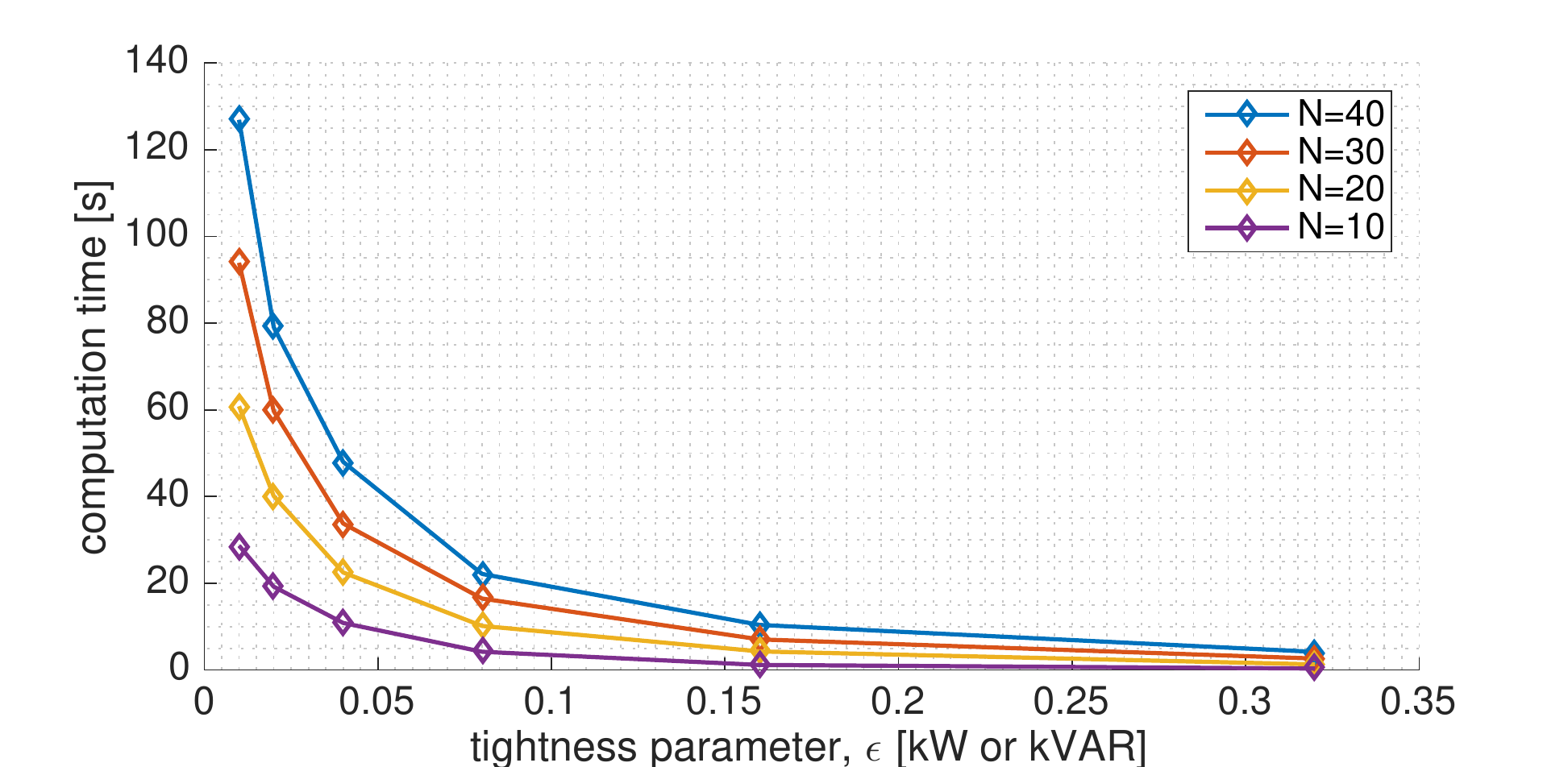}\label{F:600_eps}
}
\hspace{-0.3in}
\subfigure[computation time vs. ensemble size]{
\includegraphics[scale=0.42]{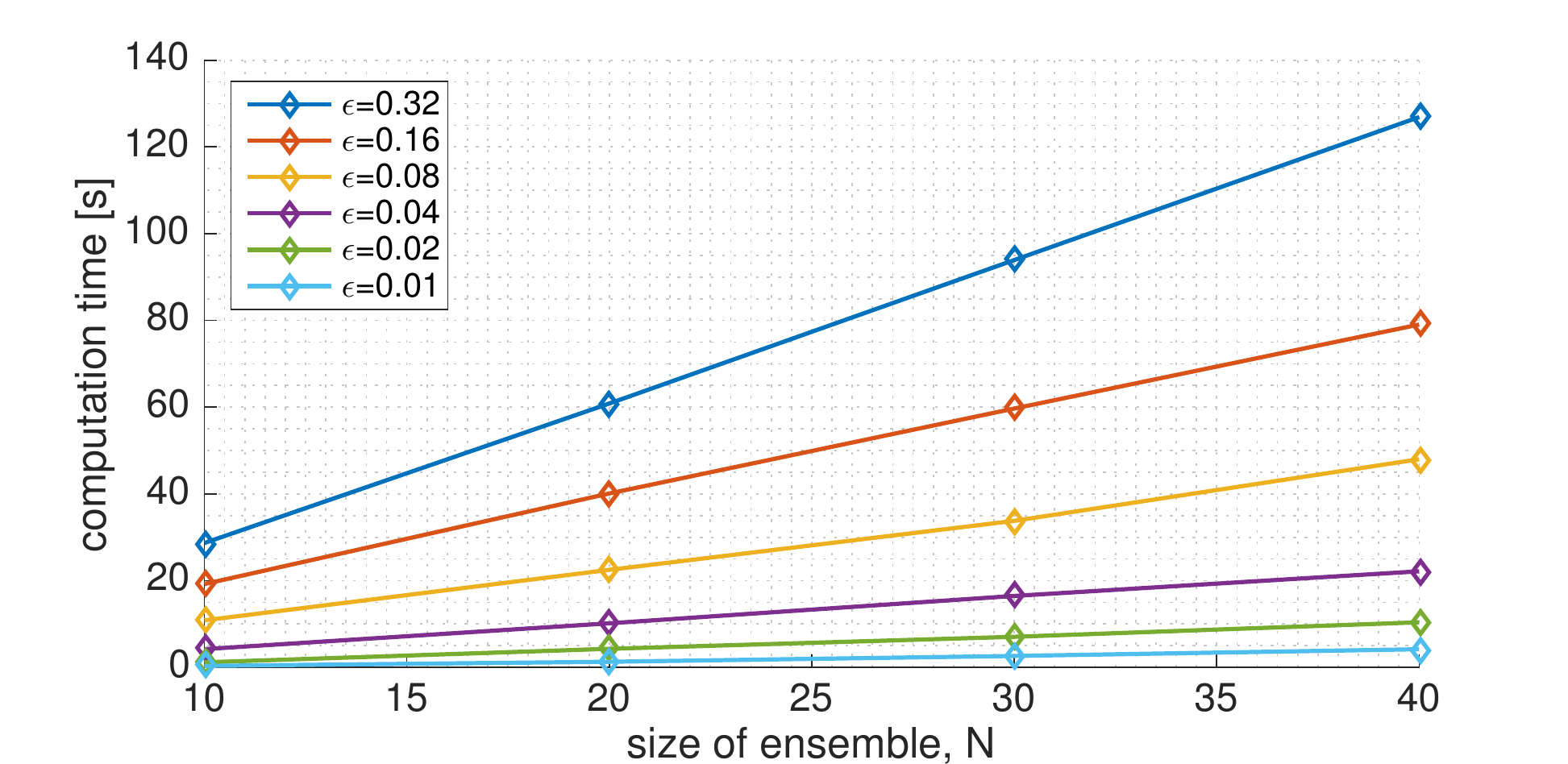}\label{F:600_N}
}
\caption[Optional caption for list of figures]{Computational complexity with respect to ensemble size ($N$) and tightness parameter ($\varepsilon$), under the scenario \eqref{E:halfbound} when $\overline{M}_p$ is set to 600 (small).}
\label{F:600}
\end{figure*}
\begin{figure*}[h]
\centering
\captionsetup{justification=centering}
\subfigure[computation time vs. tightness parameter]{
\includegraphics[scale=0.42]{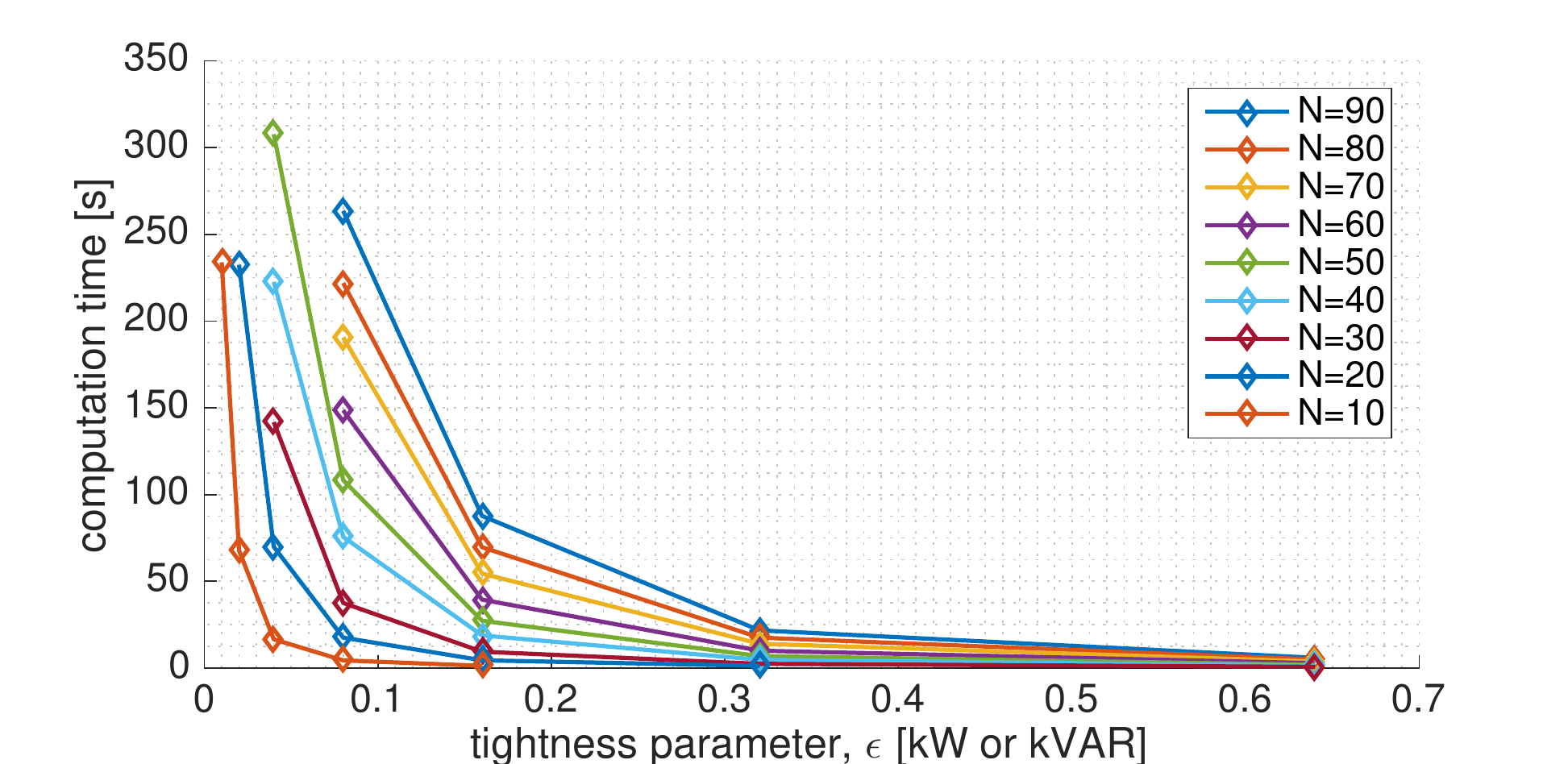}\label{F:Inf_eps}
}
\hspace{-0.3in}
\subfigure[computation time vs. ensemble size]{
\includegraphics[scale=0.42]{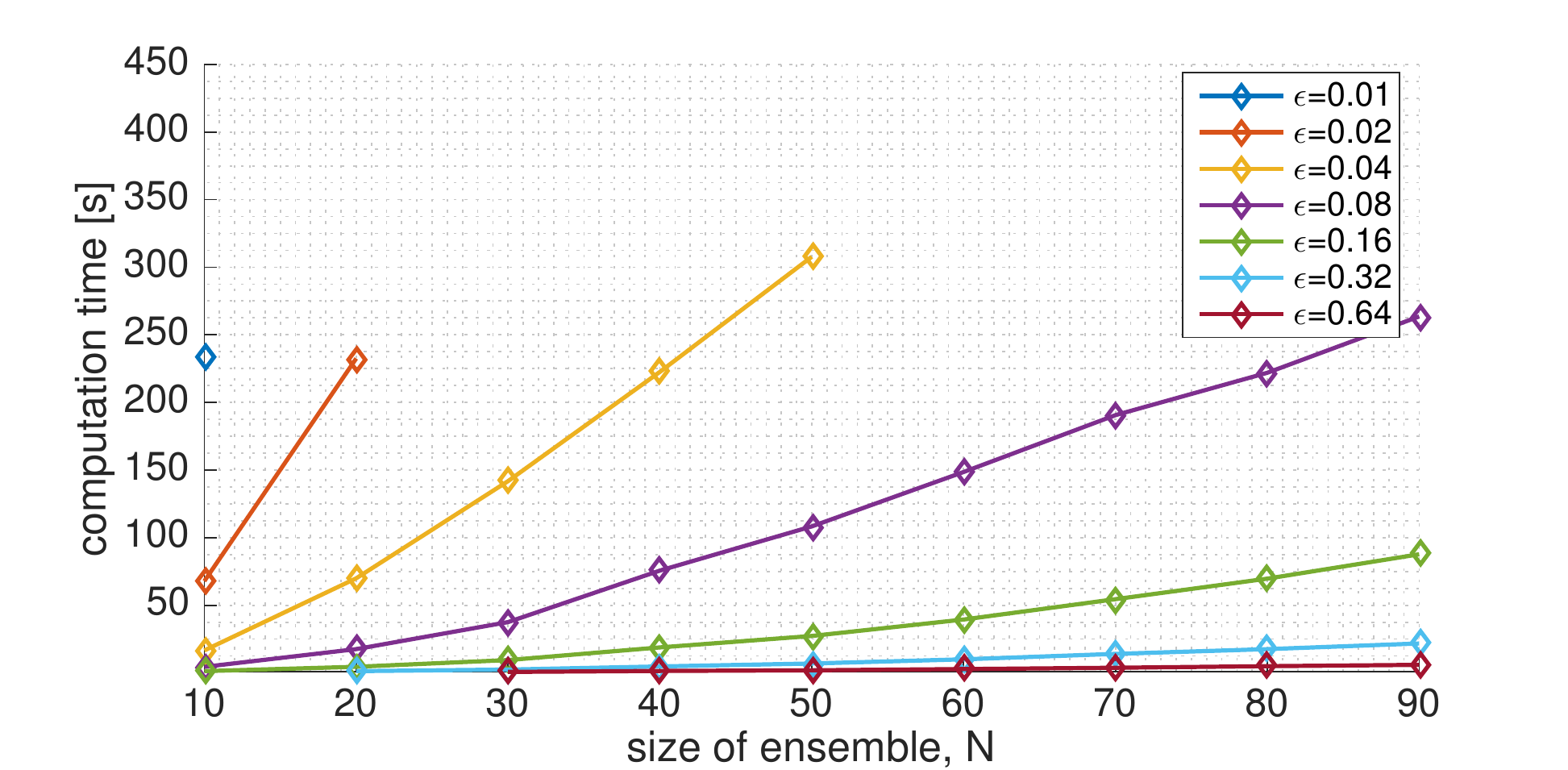}\label{F:Inf_N}
}
\caption[Optional caption for list of figures]{Computational complexity with respect to ensemble size ($N$) and tightness parameter ($\varepsilon$), under the scenario \eqref{E:halfbound} when $\overline{M}_p$ is set to 4000 (large).}
\label{F:Inf}
\end{figure*}

Next, we generate representative DER scenarios, that represent real world test cases to demonstrate the proposed DER flexibility aggregation methodology. The proposed approach entails the following three steps:

\subsection{Identify contributing factors}

\vspace{-0.1in}

We include five different types of DERs that were discussed previously - heating, ventillation and air-conditioning (HVAC) units, electric water heaters (EWHs), solar photovoltaic (PV) inverters, wind generators, battery electric vehicles (BEVs). We limit ourselves to the consideration of three factors - income level, climate type, and economic policies (with regard to incentives for renewables) - that can potentially affect the penetration of each of these DERs in a certain geographical region. In particular, we assume that the income level and climate type affect the penetration of HVAC and EWHs. For renewable DERs, i.e. PV and wind generators, we assume that climate type and level of incentives determine their prevalence. Lastly, for BEVs, we consider income level and level of incentives as the primary factors affecting their penetration. 

We create combinations of the three factors ident- ified above, such that each combination represents a DER scenario. For this we consider the following bi- nary variations of each of these factors: 1) climate type - \textit{mild or extreme}; 2) income level: \textit{high or low}; and 3) incentive level: \textit{high or low}.
To simplify the analysis further, we assume that the climate type classification is with regard to heat and humidity, and correspondingly the HVAC type considered refers to space cooling devices. The above variations result in a total of $2^3 = 8$ scenarios.  

\subsection{Generate a DER mix per scenario}

\vspace{-0.1in}

For each of the above scenarios, we identify a relative mix of DERs expressed as a percentage assignment for each type of DER present in the mix. We use several previous studies and surveys to estimate these assigments, representing a futuristic scenario with a generally high level of DERs. For HVAC, we use data from \textit{Residential Energy Consumption Survey} (RECS) \cite{recs} for quantifying the impact of income level, and the climate zone definitions in the Builing America (BA) Program \cite{ba} for climate type classification. For renewable DERs (solar and wind), we use data from the \textit{Annual Energy Outlook} (AOE) 2017 \cite{aeo} which provides projected levels of solar and wind capacity installations expected to be in place by 2050, for 22 US regions classified by the the \textit{North American Electric Reliability Corporation} (NERC) \cite{ner}. For each of these regions, the climate type classification was based on BA, while the incentive level was assigned based on the \textit {Database of State Incentives for Renewables and Efficiency} (DSIRE) \cite{dsire}. For EWHs, similar to HVAC, data from RECS and BA was used. For BEVs, data from \cite{net-benefit} was used to determine the incentive level for 15 states in the US. The income level for each of these states was obtained from \cite{census}. These factors were then mapped to projected sales of BEVs in these states in 2040 based on \cite{forbes}. 
The resulting scenarios are shown in Table\,\ref{prelim-scenarios}. In Table\,\ref{prelim-scenarios}, the percentages for each DER should be interpreted as follows: 1) {[HVAC, EWH and BEV]: percentage of consumers which use these resources; 2) {[solar and wind]: installed capacity as a percentage of total installed electricity capacity. We use the fact that the total installed electric capacity in the US is about 1074.64 GW \cite{eia-cap}.}



\begin{table*}[thb]
\caption{Scenarios with varying DER distributions}
\label{prelim-scenarios}
\centering
\begin{tabular}{|l*{9}{|c}}\hline
Scenario & Income level & Climate type & Incentive level & \% HVAC & \%solar & \%wind & \%EWHs & \%BEVs  \\\hline
1 & Low	 & Mild & Low & 78.9	& 9.6 & 15.3 & 41.4 & 30.6\\\hline
2 & Low & Mild & High & 78.9 & 48.1 & 16 & 41.4 & 47.4 \\\hline
3 & Low	 & Extreme & Low & 89.8 & 8.5 & 12.6 & 55.4 & 30.6\\\hline
4 & Low	 & Extreme & High & 89.8	 & 45.8 & 13.2 & 55.4 & 47.4\\\hline
5 & High & Mild & Low & 81.7	 & 9.6 & 15.3 & 36.2 & 36.2\\\hline
6 & High & Mild & High & 81.7  & 48.1 & 16 & 36.2 & 51.5\\\hline
7 & High & Extreme & Low & 92.6 & 8.5 & 12.6 & 50.2 & 36.2\\\hline
8 & High & Extreme & High & 92.6 & 45.8 & 13.2 & 50.2 & 51.5\\\hline
\end{tabular}
\end{table*}

We use the procedure described in this paper to approximate the aggregated flexibility associated with various DER scenarios in Table\,\ref{prelim-scenarios}. Using a total population size of 20 DERs, we generate random ensembles based on the percentages from the Table\,\ref{prelim-scenarios}. Specifically, we show the plots for scenarios 1, 3, 7 and 8\,. Comparing between the scenarios 7 and 8, we see that the increase in penetration of solar PVs show up, for example, as a sharper boundary on the right-hand side. On the other hand, the scenario 3 has a distinctive narrower affine shape, which can be attributed to its relatively higher penetration of air-conditioning loads compared to other loads.
\begin{figure*}[h]
\centering
\captionsetup{justification=centering}
\subfigure[scenario 1]{
\includegraphics[scale=0.42]{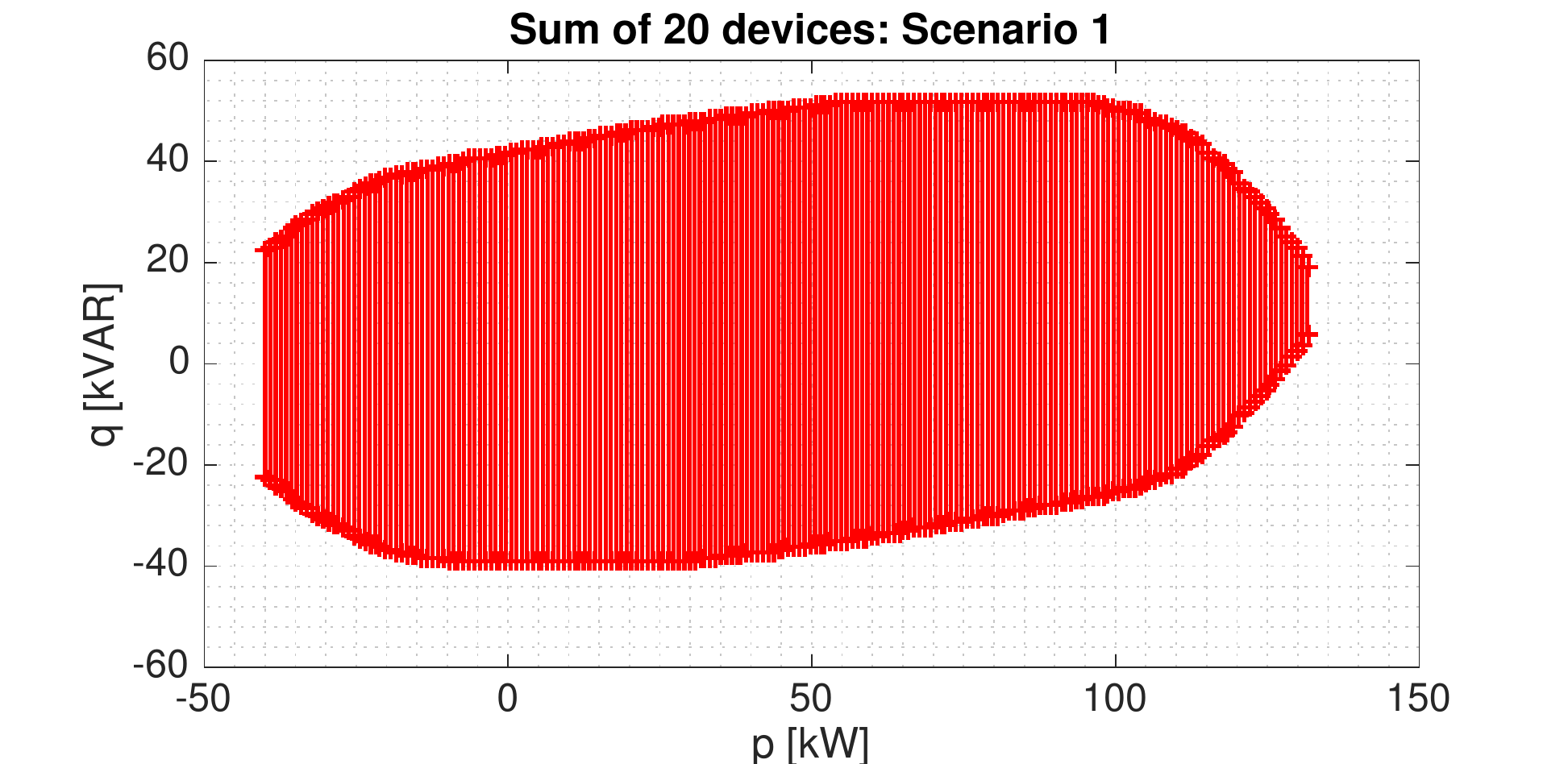}\label{F:scen1}
}
\hspace{-0.3in}
\subfigure[scenario 3]{
\includegraphics[scale=0.42]{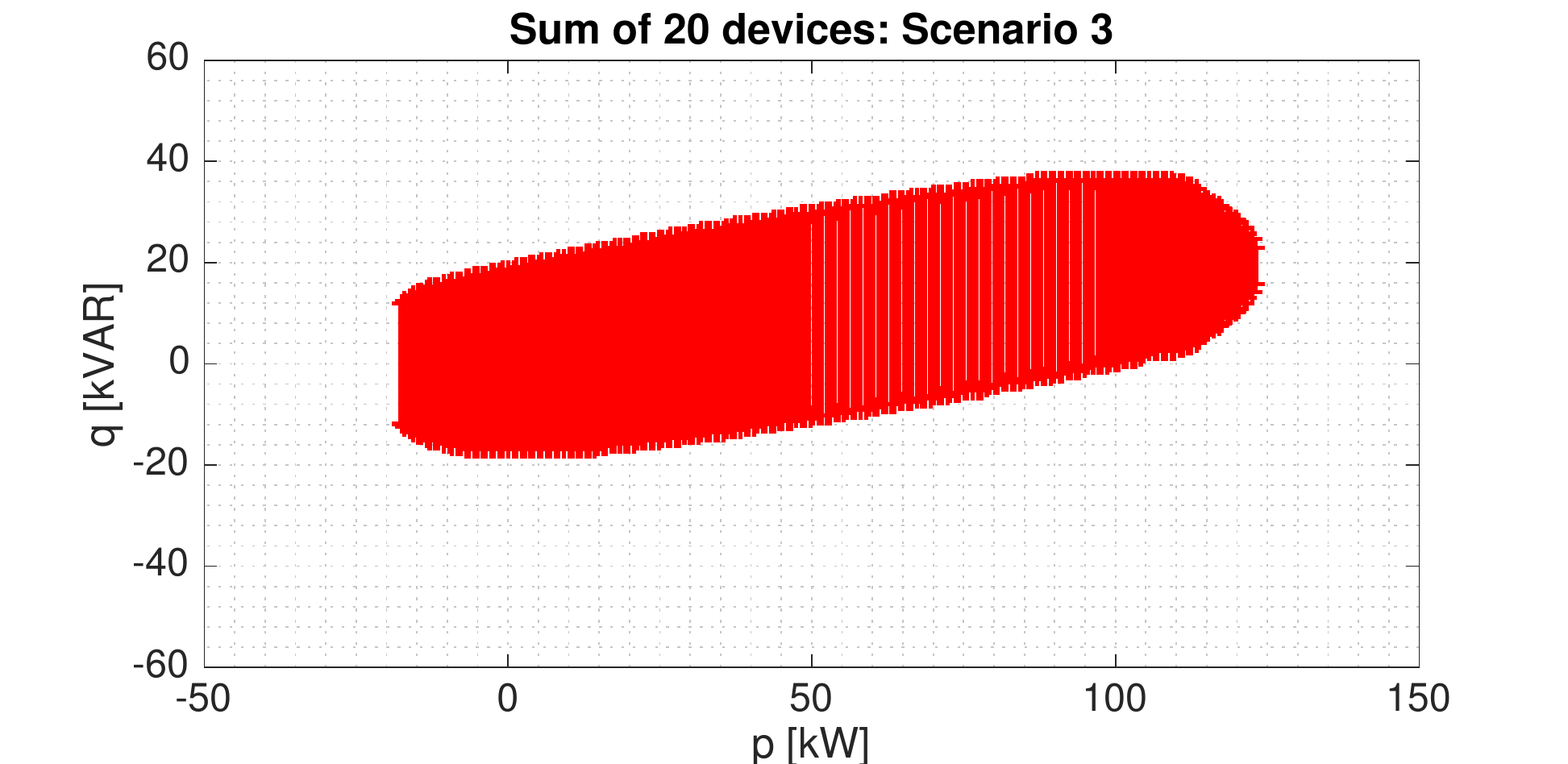}\label{F:scen3}
}
\subfigure[scenario 7]{
\includegraphics[scale=0.42]{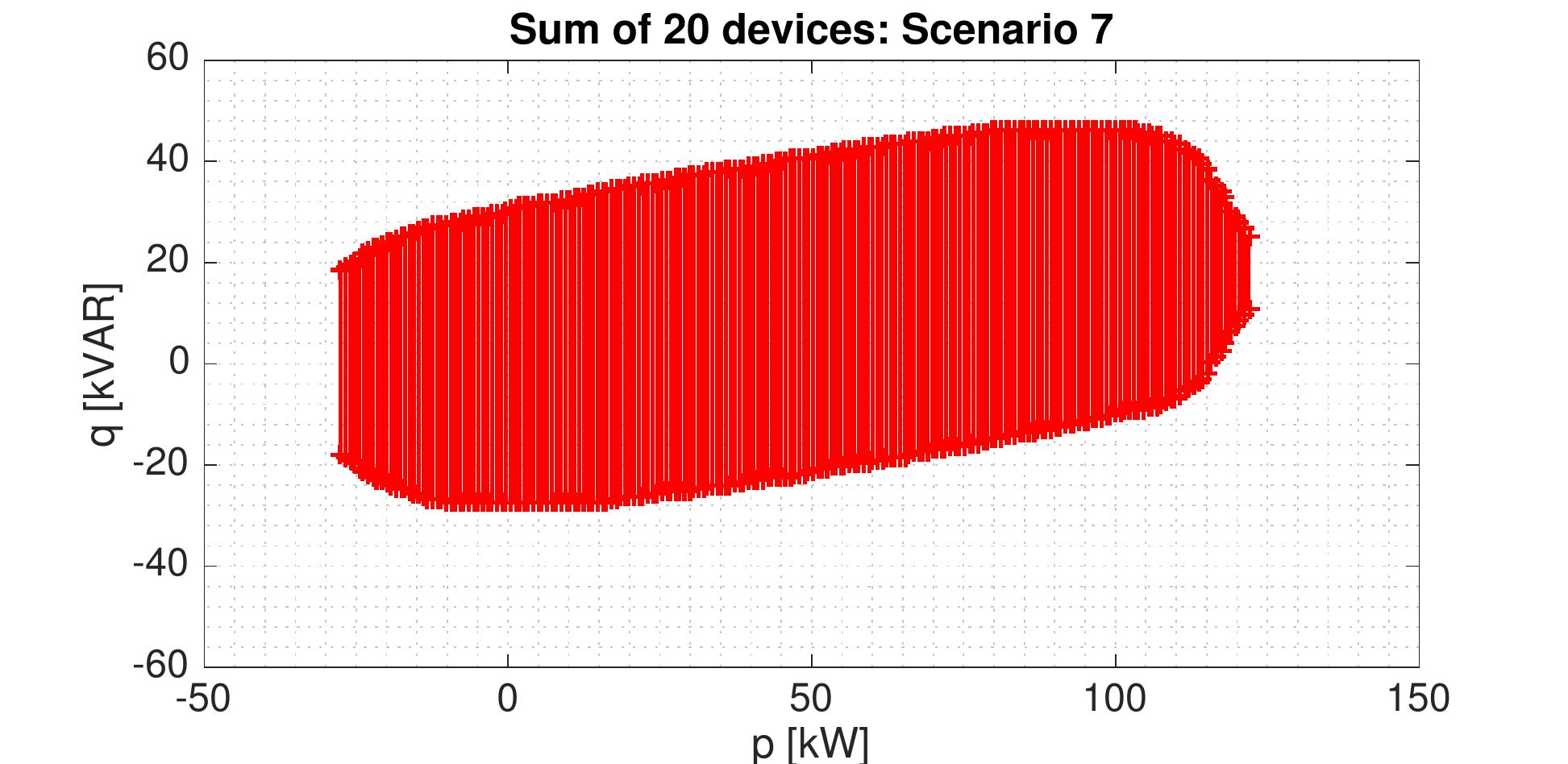}\label{F:scen7}
}
\hspace{-0.3in}
\subfigure[scenario 8]{
\includegraphics[scale=0.42]{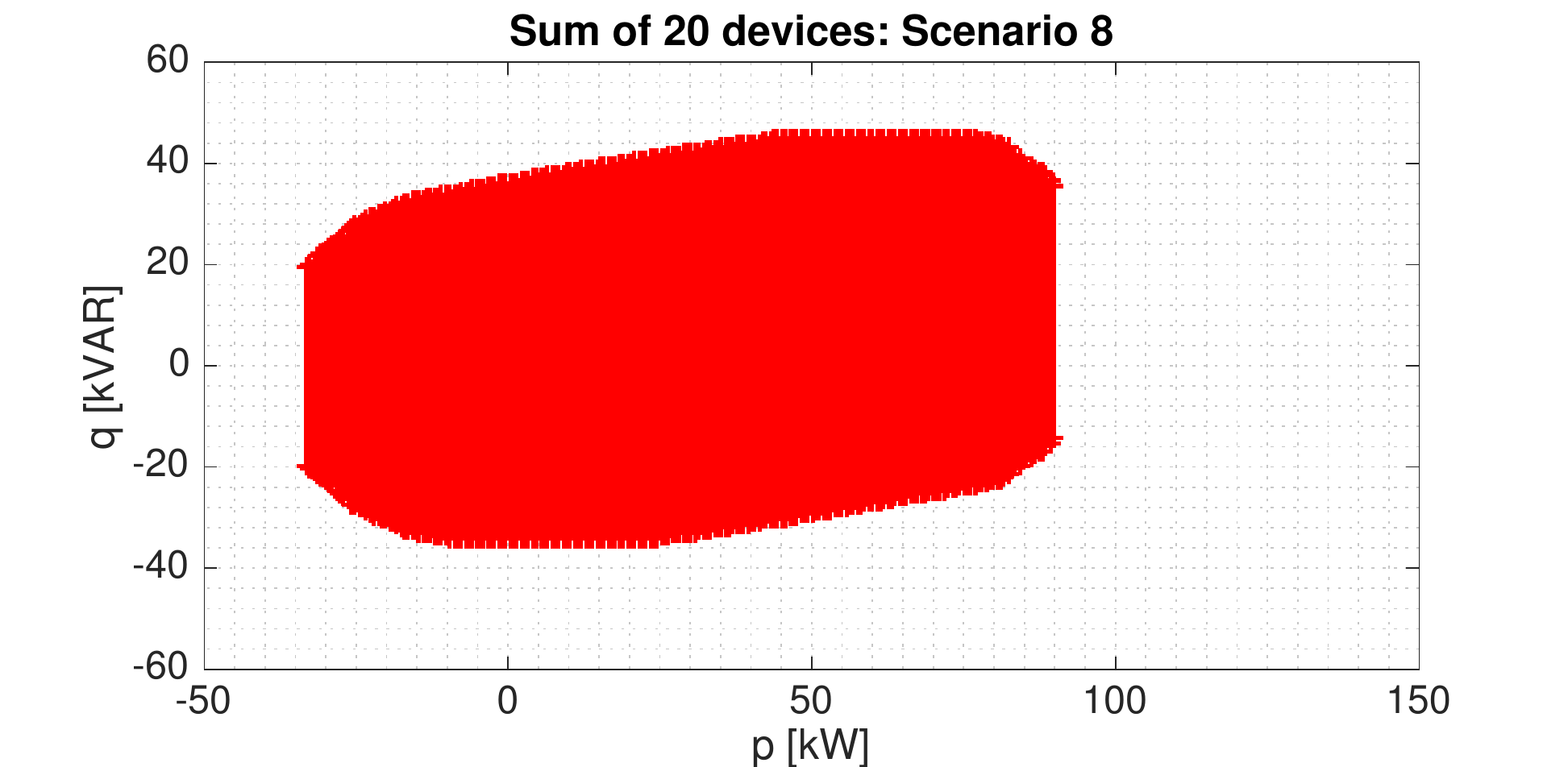}\label{F:scen8}
}
\caption[Optional caption for list of figures]{Aggregated flexibility domains for some representative DER scenarios listed in Table\,\ref{prelim-scenarios}.}
\label{F:scens}
\end{figure*}

\section{Conclusions}\label{S:concl}

\vspace{-0.1in}

This paper is motivated from the need for distributed control and optimization at the power distribution networks. Specifically, the problem of modeling short- term flexibility in ensembles of heterogeneous DERs is addressed. We present a scalable algorithm to approximate, within some specified error tolerance, the aggregated short-term flexibility in an ensemble of DERs. Performance analysis of the proposed algo- rithm is presented in terms of complexity and accuracy of the results. We show that the algorithm can ac- hieve $\mathcal{O}(N)$ and $\mathcal{O}(N^2)$ complexity in some special scenarios, while achieving $\mathcal{O}(N^3)$ complexity in likely realistic scenarios. Numerical results are provided to demonstrate the applicability of the algorithm.

\section*{Acknowledgment}

\vspace{-0.1in}

This work was supported by the US Department of Energy under
the Grid Modernization Lab Consortium initiative (contract DE-AC02-76RL01830).






\bibliographystyle{ieeetr}
\bibliography{RefList,references,MyReferences,RefMinkowski}

\end{document}